\documentclass{amsart}
\usepackage{amsfonts}
\usepackage{amssymb}
\usepackage{amsmath}

\newtheorem{theorem}{Theorem}[section]
\theoremstyle{plain}

\newtheorem{definition}{Definition}[section]
\newtheorem{example}{Example}[section]

\newtheorem{lemma}{Lemma}[section]

\numberwithin{equation}{section}
\input{tcilatex}

\begin{document}
\title[Stability and convergence in convex $A$-metric spaces]{On the
stability and convergence of Mann iteration process in convex $A$-metric
spaces}
\dedicatory{{\footnotesize (Dedicated to Assoc. Prof. Dr. Birol Gunduz who
passed away on the 3rd of April, 2019.)}}
\author{Isa Yildirim}
\address{Department of Mathematics, Faculty of Science, Ataturk University,
Erzurum, 25240, Turkey.}
\email{isayildirim@atauni.edu.tr}

\begin{abstract}
In this paper, firstly, we introduce the concept of convexity in $A$-metric
spaces and show that Mann iteration process converges to the unique fixed
point of Zamfirescu type contractions in this newly defined convex $A$%
-metric space. Secondly, we define the concept of stability in convex $A$%
-metric spaces and establish stability result for the Mann iteration process
considered in such spaces. Our results carry some well-known results from
the literature to convex $A$-metric spaces.
\end{abstract}

\maketitle

\section{\textbf{Introduction and preliminaries}}

The Banach Fixed Point Theorem which is the one of the most important
theorem in all analysis. It plays a key role for many applications in
nonlinear analysis. For example, in the areas such as optimization,
mathematical models, and economic theories. Due to this, the result has been
generalized in various directions. As a generalization of metric space,
Mustafa and Sims introduced a new class of generalized metric spaces called $%
G$-metric spaces (see \cite{z12}, \cite{z13}) as a generalization of metric
spaces $(X,d).$ This was done to introduce and develop a new fixed point
theory for a variety of mappings in this new setting. This helped to extend
some known metric space results to this more general setting. The $G$-metric
space is defined as follows:

\begin{definition}
\cite{z13} Let $X$ be a nonempty set and let $G:X\times X\times X\rightarrow 
\mathbb{R}
^{+}$ be a function satisfying the following properties:

(i) $G(x,y,z)=0$ if $x=y=z$

(ii) $0<G(x,x,y)$ for all $x,y\in X,$ with $x\neq y$

(iii) $G(x,x,y)\leq G(x,y,z)$ for all $x,y,z\in X,$ with $z\neq y$

(iv) $G(x,y,z)=G(x,z,y)=G(y,z,x)=...,$ (symmetry in all three variables); and

(v) $G(x,y,z)\leq G(x,a,a)+G(a,y,z)$ for all $x,y,z,a\in X$ (rectangle
inequality )$.$

Then the function $G$ is called a generalized metric or more specifically, a 
$G$-metric on $X$, and the pair $(X,G)$ is called a $G$-metric space.
\end{definition}

Mustafa et al. studied many fixed point results for a self-mapping in $G$%
-metric space. \cite{z3}-\cite{z1} can be cited for reference.

On the other hand, Abbas et al. \cite{a} introduced the concept of an $A$%
-metric space as follows:

\begin{definition}
Let $X$ be nonempty set. Suppose a mapping $A:X^{t}\rightarrow 
\mathbb{R}
$ satisfy the following conditions:

$\left( A_{1}\right) $ $A(x_{1},x_{2},...,x_{t-1},x_{t})\geq 0$ $,$

$\left( A_{2}\right) $ $A(x_{1},x_{2},...,x_{t-1},x_{t})=0$ if and only if $%
x_{1}=x_{2}=...=x_{t-1}=x_{t},$

$\left( A_{3}\right) $ $A(x_{1},x_{2},...,x_{t-1},x_{t})\leq
A(x_{1},x_{1},...,(x_{1})_{t-1},y)+A(x_{2},x_{2},...,(x_{2})_{t-1},y)+$

$%
...+A(x_{t-1},x_{t-1},...,(x_{t-1})_{t-1},y)+A(x_{t},x_{t},...,(x_{t})_{t-1},y) 
$

for any $x_{i},y\in X,$ $(i=1,2,...,t).$ Then, $(X,A)$ is said to be an $A$%
-metric space.
\end{definition}

It is clear that the an $A$-metric space for $t=2$ reduces to ordinary
metric $d$. Also, an $A$-metric space is a generalization of the $G$-metric
space.

\begin{example}
\cite{a} Let $X=%
\mathbb{R}
$. Define a function $A:X^{t}\rightarrow 
\mathbb{R}
$ by%
\begin{eqnarray*}
A(x_{1},x_{2},...,x_{t-1},x_{t}) &=&\left\vert x_{1}-x_{2}\right\vert
+\left\vert x_{1}-x_{3}\right\vert +...+\left\vert x_{1}-x_{t}\right\vert \\
&&+\left\vert x_{2}-x_{3}\right\vert +\left\vert x_{2}-x_{4}\right\vert
+...+\left\vert x_{2}-x_{t}\right\vert \\
&&\vdots \\
&&+\left\vert x_{t-2}-x_{t-1}\right\vert +\left\vert
x_{t-2}-x_{t}\right\vert +\left\vert x_{t-1}-x_{t}\right\vert \\
&=&\sum_{i=1}^{t}\sum_{i<j}\left\vert x_{i}-x_{j}\right\vert \text{.}
\end{eqnarray*}

Then $(X,A)$ is an $A$-metric space..
\end{example}

\begin{lemma}
\label{l2} \cite{a} Let $\left( X,A\right) $ be $A$-metric space. Then $%
A\left( x,x,\dots ,x,y\right) =A\left( y,y,\dots ,y,x\right) $ for all $%
x,y\in X$.
\end{lemma}

\begin{lemma}
\label{l3} \cite{a} Let $\left( X,A\right) $ be $A$-metric space. Then for
all for all $x,y\in X$ we have $A\left( x,x,\dots ,x,z\right) \leq \left(
t-1\right) A\left( x,x,\dots ,x,y\right) +A\left( z,z,\dots ,z,y\right) $
and $A\left( x,x,\dots ,x,z\right) \leq \left( t-1\right) A\left( x,x,\dots
,x,y\right) +A\left( y,y,\dots ,y,z\right) $.
\end{lemma}

\begin{definition}
\cite{a} Let $\left( X,A\right) $ be $A$-metric space.

(i) A sequence $\left\{ x_{n}\right\} $ in $X$ is said to converge to a
point $u\in X$ if $A\left( x_{n},x_{n},\dots ,x_{n},u\right) \rightarrow 0$
as $n\rightarrow \infty $.

(ii) A sequence $\left\{ x_{n}\right\} $ in $X$ is called a Cauchy sequence
if $A\left( x_{n},x_{n},\dots ,x_{n},u_{m}\right) \rightarrow 0$ as $%
n,m\rightarrow \infty $.

(iii) The $A$-metric space $\left( X,A\right) $ is said to be complete if
every Cauchy sequence in $X\ $is convergent.
\end{definition}

Recently, Yildirim \cite{i} introduced the notion of Zamfirescu mappings in $%
A$-metric space as follows:

\begin{definition}
\label{d1} Let $\left( X,A\right) $ be $A$-metric space and $f:X\rightarrow
X $ be a mapping. $f$ is called a $A$-Zamfirescu mapping ($AZ$ mapping), if
and only if, there are real numbers, $0\leq a<1$, $0\leq b,c<\frac{1}{t}$
such that for all $x,y\in X$, at least one of the next conditions is true: 
\begin{equation*}
\left( AZ_{1}\right) A(fx,fx,\dots ,fx,fy)\leq aA(x,x,\dots ,x,y)
\end{equation*}%
\begin{equation*}
\left( AZ_{2}\right) A(fx,fx,\dots ,fx,fy)\leq b\left[ A(fx,fx,\dots
,fx,x)+A(fy,fy,\dots ,fy,y)\right]
\end{equation*}%
\begin{equation*}
\left( AZ_{3}\right) A(fx,fx,\dots ,fx,fy)\leq c\left[ A(fx,fx,\dots
,fx,y)+A(fy,fy,\dots ,fy,x)\right]
\end{equation*}
\end{definition}

Yildirim \cite{i} also extended the Zamfirescu results \cite{16} to $A$%
-metric spaces and he obtained the following results on fixed point theorems
for such mappings.

\begin{lemma}
\label{l1} \cite{i} Let $\left( X,A\right) $ be $A$-metric space and $%
f:X\rightarrow X$ be a mapping. If $f$ is a $AZ$ mapping, then there is $%
0\leq \delta <1$ such that 
\begin{equation}
A\left( fx,fx,\dots ,fx,fy\right) \leq \delta A\left( x,x,\dots ,x,y\right)
+t\delta A(fx,fx,\dots ,fx,x)  \label{1}
\end{equation}%
and 
\begin{equation}
A\left( fx,fx,\dots ,fx,fy\right) \leq \delta A\left( x,x,\dots ,x,y\right)
+t\delta A(fy,fy,\dots ,fy,x)  \label{2}
\end{equation}%
for all $x,y\in X$.
\end{lemma}

\begin{theorem}
\label{A} \cite{i} \noindent Let $\left( X,A\right) $ be complete $A$-metric
space and $f:X\rightarrow X$ be an$\ AZ$ mapping. Then $f$ has a unique
fixed point and Picard iteration process $\left\{ x_{n}\right\} $ defined by 
$x_{n+1}=fx_{n}$ converges to a fixed point of $f$.
\end{theorem}

Studies in metric spaces are related to the existence of fixed point without
approximating them. The reason behind is the unavailablity of convex
structure in metric spaces. To solve this problem, Takahashi \cite{Ta}
introduced the notion of convex metric spaces and studied the approximation
of fixed points for nonexpansive mappings in this setting. Inspired by this,
Yildirim and Khan \cite{yk} defined convex structure in $G$-metric spaces
and they transformed the Mann iterative process to a convex $G$-metric space
as follows. And they also proved some fixed point theorems deal with
convergence of Mann iteration process for some class of mappings.

\begin{definition}
\label{yk1} \cite{yk} Let $(X,G)$ be a $G$-metric space. A mapping $%
W:X^{2}\times I^{2}\rightarrow X$ is termed as a convex structure on $X$ if $%
G(W(x,y;\lambda ,\beta ),u,v)\leq \lambda G(x,u,v)+\beta G(y,u,v)$ for real
numbers $\lambda $ and $\beta $ in $I=[0,1]$ satisfying $\lambda +\beta =1$
and $x,y,u$ and $v\in X.$

A $G$-metric space $(X,G)$ with a convex structure $W$ is called a convex $G$%
-metric space and denoted as $(X,G,W).$

A nonempty subset $C$ of a convex $G$-metric space $(X,G,W)$ is said to be
convex if $W(x,y;a,b)\in C$ for all $x,y\in C$ and $a,b\in I.$\ 
\end{definition}

\begin{definition}
\label{ykk2} \cite{yk} Let $(X,G,W)$ be convex $G$-metric space with convex
structure $W$ and $f:X\rightarrow X$ be a mapping. Let $\left\{ {\alpha }%
_{n}\right\} $ be a sequence in $[0,1]$ for $n\in 
\mathbb{N}
.$ Then for any given $x_{0}\in X,$ the iterative process defined by the
sequence $\left\{ x_{n}\right\} $ as%
\begin{equation}
x_{n+1}=W\left( x_{n},fx_{n};1-{\alpha }_{n},{\alpha }_{n}\right) ,\ \text{\
\ \ }n\in 
\mathbb{N}
,  \label{m}
\end{equation}%
is called Mann iterative process in the convex metric space $(X,G,W).$
\end{definition}

The iterative approximation of a fixed point for certain classes of mappings
is one of the main tools in the fixed point theory. Many authors (\cite{kd,
gund, AR, Kim, Liu, Modi, 10, 29, yil, yko}) discussed the existence of
fixed points and convergence of different iterative processes for various
mappings in convex metric spaces.

Keeping the above in mind, in this paper, we first define the concept of
convexity in $A$-metric spaces. Then, we use Mann iteration in this newly
defined convex $A$-metric space to prove some convergence results for
approximating fixed points of some classes of mappings. We also discuss
stability result for the Mann iteration process. Results in this paper show
that different iteration methods can be used to approximate fixed points of
different class of mappings in $A$-metric spaces. Our results are just new
in the setting.

Now, we define convex structure in $A$-metric spaces as follows.

\begin{definition}
\label{yy} Let $(X,A)$ be a $A$-metric space and $I=\left[ 0,1\right] $. A
mapping $W:X^{t}\times I^{t}\rightarrow X$ is termed as a convex structure
on $X$ if 
\begin{eqnarray}
&&A(u_{1},u_{2},...,u_{t-1},W(x_{1},x_{2},...,x_{t-1},x_{t};a_{1},a_{2},...,a_{t}))
\label{m9} \\
&\leq
&a_{1}A(u_{1},u_{2},...,u_{t-1},x_{1})+a_{2}A(u_{1},u_{2},...,u_{t-1},x_{2})
\notag \\
&&+...+a_{t}A(u_{1},u_{2},...,u_{t-1},x_{t})  \notag \\
&=&\sum_{i=1}^{t}a_{i}A(u_{1},u_{2},...,u_{t-1},x_{i})  \notag
\end{eqnarray}%
for real numbers $a_{1},a_{2},...,a_{t}$ in $I=[0,1]$ satisfying $%
\sum_{i=1}^{t}a_{i}=1$ and $u_{i}$, $x_{i}\in X$ for all $i=1,2,...,t$.

An $A$-metric space $(X,A)$ with a convex structure $W$ is called a convex $%
A $-metric space and denoted as $(X,A,W).$

A nonempty subset $C$ of a convex $A$-metric space $(X,A,W)$ is said to be
convex if

$W(x_{1},x_{2},...,x_{t-1},x_{t};a_{1},a_{2},...,a_{t})\in C$ for all $%
x_{i}\in C$ and $a_{i}\in I$, $i=1,2,...,t$.
\end{definition}

Next, we transform the Mann iteration process to a convex $A$-metric space
as follows.

\begin{definition}
\label{yk2} Let $(X,A,W)$ be convex $A$-metric space with convex structure $%
W $ and $f:X\rightarrow X$ be a mapping. Let $\left\{ {\alpha }%
_{i}^{n}\right\} $ be sequences in $[0,1]$ for all $i=1,2,...,t$ and $n\in 
\mathbb{N}
.$ Then for any given $x_{0}\in X,$ the iteration process defined by the
sequence $\left\{ x_{n}\right\} $ as%
\begin{equation}
x_{n+1}=W\left( x_{n},x_{n},...,x_{n},fx_{n};{\alpha }_{1}^{n},{\alpha }%
_{2}^{n},...,{\alpha }_{t}^{n}\right) ,  \label{m7}
\end{equation}%
is called Mann iteration process in the convex metric space $(X,A,W).$
\end{definition}

It follows from the structure of convex $A$-metric space that%
\begin{eqnarray}
A(x_{n+1},u_{1},u_{2},...,u_{t-1}) &=&A(W\left( x_{n},x_{n},...,x_{n},fx_{n};%
{\alpha }_{1}^{n},{\alpha }_{2}^{n},...,{\alpha }_{t}^{n}\right)
,u_{1},u_{2},...,u_{t-1})  \label{m8} \\
&\leq &{\alpha }_{1}^{n}A(x_{n},u_{1},u_{2},...,u_{t-1})+{\alpha }%
_{2}^{n}A(x_{n},u_{1},u_{2},...,u_{t-1})  \notag \\
&&+...+{\alpha }_{t}^{n}A(fx_{n},u_{1},u_{2},...,u_{t-1}).  \notag
\end{eqnarray}

If we take $t=2$ in $(\ref{m7})$ and $(\ref{m8})$, this structures reduce to 
$(\ref{m})$ and $(\ref{m9})$.

The following Lemma shall be used in the proof of the stability result.

\begin{lemma}
\label{mm} \cite{ber} If $\delta $ is a real number such that $0\leq \delta
<1$ and $\left\{ {\varepsilon }_{n}\right\} $ is a sequence of positive
numbers such that ${\mathop{\mathrm{lim}}_{n\rightarrow \infty }{\varepsilon 
}_{n}=0\ }$, then for any sequence of positive numbers $\left\{
u_{n}\right\} $ satisfying 
\begin{equation*}
u_{n+1}\leq \delta u_{n}+{\varepsilon }_{n}\ ,\ n=0,1,\dots
\end{equation*}%
we have 
\begin{equation*}
{\mathop{\mathrm{lim}}_{n\rightarrow \infty }u_{n}=0\ }.
\end{equation*}
\end{lemma}

\section{\textbf{Main Results}}

\textbf{2.1 Convergence Result: }In this section, we prove the Mann
iteration process converges to fixed point of Zamfirescu mappings in
complete convex metric space $(X,A,W)$.

\begin{theorem}
\label{main1} Let $(X,A,W)$ be a complete convex $A$-metric space with a
convex structure $W$ and, $f:X\rightarrow X$ be an $AZ$ mapping$.$ Let $%
\left\{ x_{n}\right\} $ be defined iteratively by $(\ref{m7})$ and $x_{0}\in
X,$ with $\left\{ {\alpha }_{t}^{n}\right\} \subset \lbrack
0,1],\tsum\limits_{i=1}^{t}a_{i}=1$ satisfying $\tsum\limits_{n=0}^{\infty }{%
\alpha }_{t}^{n}=\infty $ for all $n\in 
\mathbb{N}
$ and $i=1,2,...,t.$ Then $\left\{ x_{n}\right\} $ converges to a unique
fixed point of $f.$

\begin{proof}
From Theorem \ref{A}, we know that an$\ AZ$ mapping has a unique fixed point
in $X$. Call it $u$ and consider $x_{i}\in X$, $i=1,2,...,t$.

At least one of $\left( AZ_{1}\right) $, $\left( AZ_{2}\right) $ and $\left(
AZ_{3}\right) $ is satisfied. If $\left( AZ_{1}\right) ,\left( AZ_{2}\right) 
$ or $\left( AZ_{3}\right) $ holds, we know that the following inequality
from Lemma \ref{l1}%
\begin{equation}
A\left( fx,fx,\dots ,fx,fy\right) \leq \delta A\left( x,x,\dots ,x,y\right)
+t\delta A(fx,fx,\dots ,fx,x)  \label{25}
\end{equation}%
for all $x,y\in X$.

Let $\left\{ x_{n}\right\} $ be the Mann iteration process $(\ref{m7})$,
with $x_{0}\in X$ arbitrary. Then 
\begin{eqnarray*}
A(u,u,...,u,x_{n+1}) &\leq &A(u,u,...,u,W\left( x_{n},x_{n},...,x_{n},fx_{n};%
{\alpha }_{1}^{n},{\alpha }_{2}^{n},...,{\alpha }_{t}^{n}\right) ) \\
&\leq &{\alpha }_{1}^{n}A(u,u,...,u,x_{n})+{\alpha }%
_{2}^{n}A(u,u,...,u,x_{n}) \\
&&+...+{\alpha }_{t}^{n}A(u,u,...,u,fx_{n}) \\
&=&\left( 1-{\alpha }_{t}^{n}\right) A(u,u,...,u,x_{n})+{\alpha }%
_{t}^{n}A(u,u,...,u,fx_{n}).
\end{eqnarray*}%
Take $x=u$ and $y=x_{n}$ in (\ref{25}) to obtain%
\begin{eqnarray}
A(u,u,...,u,fx_{n}) &=&A(fu,fu,...,fu,fx_{n})  \label{5} \\
&\leq &\delta A(u,u,...,u,x_{n})+t\delta A(fu,fu,\dots ,fu,u)  \notag \\
&=&\delta A(u,u,...,u,x_{n})  \notag
\end{eqnarray}%
which together with (\ref{5}) yields 
\begin{eqnarray}
A(u,u,...,u,x_{n+1}) &\leq &\left( 1-{\alpha }_{t}^{n}\right)
A(u,u,...,u,x_{n})+{\alpha }_{t}^{n}\delta A(u,u,...,u,x_{n})  \label{6} \\
&=&\left[ 1-\left( 1-\delta \right) {\alpha }_{t}^{n}\right]
A(u,u,...,u,x_{n}).  \notag
\end{eqnarray}%
Inductively we get 
\begin{eqnarray}
A(u,u,...,u,x_{n+1}) &\leq &\left[ 1-\left( 1-\delta \right) {\alpha }%
_{t}^{n}\right] A(u,u,...,u,x_{n})  \label{7} \\
&\leq &\left[ 1-\left( 1-\delta \right) {\alpha }_{t}^{n}\right] \left[
1-\left( 1-\delta \right) {\alpha }_{t}^{n-1}\right] A(u,u,...,u,x_{n-1}) 
\notag \\
&&\vdots  \notag \\
&\leq &\tprod\limits_{k=0}^{n}\left[ 1-\left( 1-\delta \right) {\alpha }%
_{t}^{k}\right] A(u,u,...,u,x_{0})  \notag
\end{eqnarray}%
As $0\leq \delta <1,\left\{ {\alpha }_{t}^{k}\right\} \subset \lbrack 0,1]$
and $\tsum\limits_{k=0}^{\infty }{\alpha }_{t}^{k}=\infty $, we have%
\begin{equation*}
\lim_{n\rightarrow \infty }\tprod\limits_{k=0}^{n}\left[ 1-\left( 1-\delta
\right) {\alpha }_{t}^{k}\right] =0,
\end{equation*}%
which by (\ref{7}) implies%
\begin{equation*}
\lim_{n\rightarrow \infty }A(u,u,...,u,x_{n+1})=\lim_{n\rightarrow \infty
}A(x_{n+1},x_{n+1},...,x_{n+1},u)=0.
\end{equation*}%
Hence the sequence $\left\{ x_{n}\right\} $ defined iteratively by $(\ref{m7}%
)$ converges to the fixed point of $f$.
\end{proof}
\end{theorem}

\textbf{2.2 Stability Result: }Now, we will give stability result for the
Mann iteration $(\ref{m7})$ in complete convex $A$-metric space.

\begin{definition}
Let $(X,A,W)$ be a convex $A$-metric space with a convex structure $W$ and, $%
f:X\rightarrow X$ be a mapping, $x_{0}\in X$ and let us assume that the
iteration process\ $(\ref{m7})$, that is, the sequence $\left\{
x_{n}\right\} $ defined by $(\ref{m7})$, converges to a fixed point $u$ of $%
f $.

Let $\left\{ y_{n}\right\} $ be an arbitrary sequence in $X$ and set%
\begin{equation*}
\epsilon _{n}=A\left( y_{n+1},y_{n+1},...,y_{n+1},g\left( f,y_{n}\right)
\right) \text{ for }n=0,1,2,...
\end{equation*}%
where $g\left( f,y_{n}\right) =W\left( y_{n},y_{n},...,y_{n},fy_{n};{\alpha }%
_{1}^{n},{\alpha }_{2}^{n},...,{\alpha }_{t}^{n}\right) $ and $\left\{ {%
\alpha }_{i}^{n}\right\} $ are real sequences in $[0,1]$ for $i=1,2,...,t$.

We say that the Mann iteration process $(\ref{m7})$ is $f$-stable or stable
with respect to $f$ if and only if%
\begin{equation*}
\lim_{n\rightarrow \infty }\epsilon _{n}=0\Longleftrightarrow
\lim_{n\rightarrow \infty }y_{n}=u.
\end{equation*}
\end{definition}

\begin{theorem}
Let $(X,A,W)$ be a complete convex $A$-metric space with a convex structure $%
W$ and, $f:X\rightarrow X$ be an $AZ$ mapping$.$Let $\left\{ x_{n}\right\} $
be defined iteratively by $(\ref{m7})$ and $x_{0}\in X,$ with $\left\{ {%
\alpha }_{t}^{n}\right\} \subset \lbrack 0,1],\tsum\limits_{i=1}^{t}a_{i}=1$
satisfying $0<\alpha \leq \alpha _{n}$ and $\tsum\limits_{n=0}^{\infty }{%
\alpha }_{t}^{n}=\infty $ for all $n\in 
\mathbb{N}
$ and $i=1,2,...,t.$ Then the Mann iteration process $(\ref{m7})$ is $f$%
-stable.
\end{theorem}

\begin{proof}
From Theorem \ref{A}, we know that $f$ has a unique fixed point. Suppose
that $u\in X$. From from Lemma \ref{l1}, we also know that%
\begin{equation}
A\left( fx,fx,\dots ,fx,fy\right) \leq \delta A\left( x,x,\dots ,x,y\right)
+t\delta A(fx,fx,\dots ,fx,x).  \label{26}
\end{equation}

Let $\left\{ y_{n}\right\} \subset X$ and $\epsilon _{n}=A\left(
y_{n+1},y_{n+1},...,y_{n+1},g\left( f,y_{n}\right) \right) $. Assume that $%
\lim_{n\rightarrow \infty }\epsilon _{n}=0$. Then, we will show that $%
\lim_{n\rightarrow \infty }y_{n}=u$. From (\ref{26}) and triangle
inequality, we get%
\begin{eqnarray}
&&A\left( y_{n+1},y_{n+1},\dots ,y_{n+1},,u\right)  \label{29} \\
&\leq &\left( t-1\right) A\left( y_{n+1},y_{n+1},\dots ,y_{n+1},,g\left(
f,y_{n}\right) \right)  \notag \\
&&+A\left( g\left( f,y_{n}\right) ,g\left( f,y_{n}\right) ,\dots ,g\left(
f,y_{n}\right) ,u\right)  \notag \\
&=&\left( t-1\right) A\left( y_{n+1},y_{n+1},\dots ,y_{n+1},,g\left(
f,y_{n}\right) \right)  \notag \\
&&+A\left( u,u,\dots ,u,g\left( f,y_{n}\right) \right)  \notag \\
&\leq &\left( t-1\right) {\varepsilon }_{n}+A\left( u,u,\dots ,u,g\left(
f,y_{n}\right) \right)  \notag \\
&=&\left( t-1\right) {\varepsilon }_{n}+A\left( u,u,\dots ,u,W\left(
y_{n},y_{n},\dots ,y_{n},fy_{n};{\alpha }_{1}^{n},{\alpha }_{2}^{n},\dots ,{%
\alpha }_{t}^{n}\right) \right)  \notag \\
&\leq &\left( t-1\right) {\varepsilon }_{n}+{\alpha }_{1}^{n}A\left(
u,u,\dots ,u,y_{n}\right) +{\alpha }_{2}^{n}A\left( u,u,\dots ,u,y_{n}\right)
\notag \\
&&+\dots +{\alpha }_{t}^{n}A\left( u,u,\dots ,u,{fy}_{n}\right)  \notag \\
&=&\left( t-1\right) {\varepsilon }_{n}+\left( 1-{\alpha }_{t}^{n}\right)
A\left( u,u,\dots ,u,y_{n}\right) +{\alpha }_{t}^{n}A\left( u,u,\dots ,u,{fy}%
_{n}\right)  \notag \\
&\leq &\left( t-1\right) {\varepsilon }_{n}+\left( 1-{\alpha }%
_{t}^{n}\right) A\left( u,u,\dots ,u,y_{n}\right)  \notag \\
&&+{\alpha }_{t}^{n}\left[ \delta A\left( u,u,\dots ,u,y_{n}\right) +t\delta
A\left( fu,fu,\dots ,fu,u\right) \right]  \notag \\
&=&\left[ 1-\left( 1-\delta \right) {\alpha }_{t}^{n}\right] A\left(
y_{n},y_{n},\dots ,y_{n},u\right) +\left( t-1\right) {\varepsilon }_{n}. 
\notag
\end{eqnarray}

Since $0\leq 1-\left( 1-\delta \right) {\alpha }_{t}^{n}<1-\left( 1-\delta
\right) \alpha <1$, using Lemma \ref{mm} in (\ref{29}) yields%
\begin{equation*}
{\mathop{\mathrm{lim}}_{n\rightarrow \infty }A\left( y_{n},y_{n},\dots
,y_{n},u\right) =0,\ }
\end{equation*}%
that is,%
\begin{equation*}
{\mathop{\mathrm{lim}}_{n\rightarrow \infty }y_{n}=u\ }.
\end{equation*}%
Conversely, let ${\mathop{\mathrm{lim}}_{n\rightarrow \infty }y_{n}=u\ }$.
Then,%
\begin{eqnarray*}
{\varepsilon }_{n} &=&A\left( y_{n+1},y_{n+1},\dots ,y_{n+1},,g\left(
f,y_{n}\right) \right) \\
&\leq &\left( t-1\right) A\left( y_{n+1},y_{n+1},\dots ,y_{n+1},,u\right)
+A\left( u,u,\dots ,u,g\left( f,y_{n}\right) \right) \\
&=&\left( t-1\right) A\left( y_{n+1},y_{n+1},\dots ,y_{n+1},,u\right) \\
&&+A\left( u,u,\dots ,u,W\left( y_{n},y_{n},\dots ,y_{n},fy_{n};{\alpha }%
_{1}^{n},{\alpha }_{2}^{n},\dots ,{\alpha }_{t}^{n}\right) \right) \\
&\leq &\left( t-1\right) A\left( y_{n+1},y_{n+1},\dots ,y_{n+1},,u\right) +{%
\alpha }_{1}^{n}A\left( u,u,\dots ,u,y_{n}\right) \\
&&+{\alpha }_{2}^{n}A\left( u,u,\dots ,u,y_{n}\right) +\dots +{\alpha }%
_{t}^{n}A\left( u,u,\dots ,u,{fy}_{n}\right) \\
&=&\left( t-1\right) A\left( y_{n+1},y_{n+1},\dots ,y_{n+1},,u\right)
+\left( 1-{\alpha }_{t}^{n}\right) A\left( u,u,\dots ,u,y_{n}\right) \\
&&{+\alpha }_{t}^{n}A\left( u,u,\dots ,u,{fy}_{n}\right) \\
&\leq &\left( t-1\right) A\left( y_{n+1},y_{n+1},\dots ,y_{n+1},,u\right)
+\left( 1-{\alpha }_{t}^{n}\right) A\left( u,u,\dots ,u,y_{n}\right) \\
&&+{\alpha }_{t}^{n}\left[ \delta A\left( u,u,\dots ,u,y_{n}\right) +t\delta
A\left( fu,fu,\dots ,fu,u\right) \right] \\
&\leq &\left( t-1\right) A\left( y_{n+1},y_{n+1},\dots ,y_{n+1},,u\right)
+\left( 1-{\alpha }_{t}^{n}\right) A\left( u,u,\dots ,u,y_{n}\right) \\
&&+{\alpha }_{t}^{n}\delta A\left( u,u,\dots ,u,y_{n}\right)
\end{eqnarray*}

Letting $n\rightarrow \infty $ in the above inequality, we have $%
\lim_{n\rightarrow \infty }\epsilon _{n}=0$.
\end{proof}

\end{document}